\documentclass{article}
\usepackage[margin=1in]{geometry}
\usepackage{amssymb}
\usepackage{amsmath}
\usepackage{amsthm}
\usepackage{hyperref}
\usepackage{array}

\newtheorem{theorem}{Theorem}[section]
\newtheorem*{theorem*}{Theorem}
\newtheorem{corollary}[theorem]{Corollary}
\newtheorem{conjecture}[theorem]{Conjecture}
\newtheorem{lemma}[theorem]{Lemma}
\newtheorem{proposition}[theorem]{Proposition}

\theoremstyle{definition}{
\newtheorem{definition}[theorem]{Definition}
\newtheorem*{definition*}{Definition}

\newtheorem*{example*}{Example}

\newtheorem{notation}[theorem]{Notation}
}

\numberwithin{equation}{section}

\newcommand{\Imm}{\operatorname{Im}}
\newcommand{\gl}{\operatorname{GL}_2^+(\mathbb{Q})}

\newcommand{\slz}{\operatorname{SL}_2(\mathbb{Z})}
\newcommand{\uh}{\mathbb{H}}
\newcommand{\suq}{\subseteq}
\begin{document}
\title{Effective Andr\'e-Oort Type Results for Almost-Holomorphic Modular Functions}
\author{Haden Spence}

\maketitle

\begin{abstract}
	In this short paper we discuss a number of effective and/or explicit results of Andr\'e-Oort type for the nonholomorphic function $\chi^*$, which I have discussed in a number of other papers such as \cite{Spence2017Ext} and \cite{Spence2016}.  After working in a rather ad-hoc manner to get some good estimates on the tails of the $q$-expansions involved, we prove weak effective Andr\'e-Oort results for $\chi^*$, which mimic but are not full analogues of effective Andr\'e-Oort results known due to K\"uhne \cite{Kuehne2012}/Bilu-Masser-Zannier \cite{Bilu2013} for the classical modular function $j$.  
	
	Then we go on to discuss what we call an ``explicit'' result; that certain triples of special points cannot often be collinear, looking for an analogue of \cite{Bilu2017}.  Again we cannot get a perfect analogy, but we do prove a weaker result and discuss what remains to be proved to complete this.
	
	An important result which arises as a side-effect of the explicit calculation done here is Corollary \ref{cor:FieldEquality}, which affirms a conjecture I made in earlier papers (particularly \cite{Spence2017Ext}); that for a quadratic point $\tau$ we have $\mathbb{Q}(j(\tau))=\mathbb{Q}(\chi^*(\tau))$.  Although it appears here somewhat tangentially, it may be the most significant result in the paper.	
\end{abstract}
\bigskip

\textbf{Acknowledgements.}  As ever I would like to thank my supervisor Jonathan Pila, who has guided me with his characteristic good humour and certain judgement throughout my DPhil studies.  This short paper in particular also owes a lot to Gareth Jones, who encouraged me to pursue this and with whom I had a number of very fruitful conversations on these topics.  Since I have now left Oxford, it is also a good time to thank Oxford University and my friends and colleagues at the Mathematical Institute there for their help and friendship.  Much of this work was carried out while I was funded by an EPSRC grant; I thank the EPSRC again for their generosity throughout my DPhil.

\bigskip

\textbf{Disclaimer.} In some ways this paper may be somewhat incomplete as a result of the author's being out of academic circulation.  I welcome any comments and suggestions anyone may have on how it might be improved - though such changes may be slow in coming!

\section{Introduction}\label{sect:intro}
In work carried out during the course of my DPhil studies at the University of Oxford, I proved various results of Andr\'e-Oort type in the context of ``nonclassical modular functions''.  This last phrase, of course, is rather general and could apply to any number of functions.  The functions that seem most appropriate for the purpose, however, appear to be the \emph{quasimodular} and \emph{almost holomorphic modular} functions.

These two classes of function follow a pattern that arises often among classes of `not-quite-modular' functions.  One begins with some class of holomorphic functions which fail to be quite invariant under the action of the modular group (in this case the quasimodular functions), then by applying some sort of correction one produces a nonholomorphic function which is fully invariant under $\slz$. 

Quasimodular functions arise from the derivatives of modular forms.  This is a well-known and fairly obvious construction: when differentiating the modular law for a modular form $f$,
\[f(\gamma\tau) = (c\tau+d)^kf(\tau),\]
one gets
\[f'(\gamma\tau)(c\tau+d)^{-2} = (c\tau+d)^kf'(\tau) + ck(c\tau+d)^{k-1}f(\tau)\]
and hence
\begin{equation}\label{eqn:qmForm}f'(\gamma\tau)(c\tau+d)^{-k-2} = f'(\tau) + k\frac{c}{c\tau+d}f(\tau),\end{equation}
so that $f'$ is \emph{nearly} a modular form of weight $k+2$.

This is essentially the definition of a quasimodular form: a holomorphic function which transforms under $\slz$ like a modular form, only with an error term which is a polynomial in $\frac{c}{c\tau+d}$ with holomorphic functions as coefficients.

In this document we tend to be more concerned with the dual almost holomorphic modular functions.  Using (\ref{eqn:qmForm}) and the fact that $\Imm\gamma\tau = \Imm\tau|c\tau+d|^{-2}$, one sees that
\begin{align*}(c\tau+d)^{-k-2}\left[f'(\gamma\tau) - \frac{ik}{2}\frac{f(\gamma\tau)}{\Imm\gamma\tau}\right]&=f'(\tau) + k\frac{c}{c\tau+d}f(\tau)-\frac{ik}{2}\frac{f(\tau)}{\Imm\tau}\frac{|c\tau+d|^2}{(c\tau+d)^2}\\
&=f'(\tau)-\frac{ik}{2}\frac{f(\tau)}{(c\tau+d)\Imm\tau}\left(2ci\Imm\tau+(c\overline{\tau}+d)\right)\\
&=f'(\tau)-\frac{ik}{2}\frac{f(\tau)}{\Imm\tau}.\end{align*}
So the function
\[\widehat{f} = f' - \frac{ikf}{2\Imm},\]
the almost holomorphic dual of $f'$, transforms like a modular form of weight $k+2$.   In fact any quasimodular form can be corrected in such a way, and corrected functions of this type are called almost holomorphic modular forms.
\begin{definition}
  A function $f:\uh\to\mathbb{C}$ is called almost holomorphic if there are holomorphic functions $f_k:\uh\to\mathbb{C}$, bounded as $\Imm\tau\to\infty$, such that
  \[f(\tau)=\sum_{r=0}^{n}f_r(\tau)(\Imm\tau)^{-r}.\]
  Such a function is an almost holomorphic modular form if there is an integer $k$ such that
  \[f(\gamma\tau)=(c\tau+d)^kf(\tau)\]
  for all\footnote{Throughout this paper we ignore level structure, dealing always with the full modular group $\slz$.} $\gamma=\begin{pmatrix}a&b\\c&d\end{pmatrix}\in\slz$ and all $\tau\in\uh$.
  
  An almost holomorphic modular function is a quotient of two almost holomorphic modular forms of equal weight.
\end{definition}

Like most modular objects, quasimodular functions have $q$-expansions, that is expressions of the form
\[f(\tau)=\sum_{n=-N}^{\infty} c_nq^n,\]
where $q=e^{2\pi i\tau}$ and $c_n\in\mathbb{C}$.  This is not quite true of almost holomorphic modular functions, of course, but they can be represented instead by polynomials with $q$-expansions as coefficients; elements of $\mathbb{C}((q))[(\Imm\tau)^{-1}]$.  Though these are not strictly $q$-expansions in the traditional sense, we will refer to them as such for the remainder of the document.

\bigskip

The prototypical quasimodular form is the weight-2 quasimodular Eisenstein series $E_2$, which we write here in terms of its $q$-expansion
\[E_2(\tau) = 1 - 24\sum_{n\geq 1}\sigma_1(n)q^n,\]
where $\sigma_k$ is the sum-of-divisors function $\sigma_k(n)=\sum_{d|n}d^k$.  $E_2$ is a weight-2 quasimodular form, and can be corrected to make an almost holomorphic modular form 
\[E_2^* = E_2-\frac{3}{\pi\Imm}.\]
We will also require the $q$-expansions for the other standard Eisenstein series
\[E_4(\tau) = 1 + 240\sum_{n\geq 1}\sigma_3(n)q^n\]
and
\[E_6(\tau) = 1 - 504\sum_{n\geq 1}\sigma_5(n)q^n.\]

With these building blocks we can create a prototypical almost holomorphic modular function
\[\chi^*=\frac{E_2^*E_4E_6}{\Delta},\]
where $\Delta$ is the classical modular discriminant function $\frac{1}{1728}(E_4^3-E_6^2)$.  There is nothing particularly special about $\chi^*$ except that, since $\Delta$ is known to be non-vanishing within $\uh$, $\chi^*$ has no singularities and is everywhere real analytic.  Moreover the field $F^*$ of almost holomorphic modular functions satisfies $F^*=\mathbb{C}(j,\chi^*)$; though this last is hardly unique to $\chi^*$.  Here $j$ is the classical modular $j$-invariant: we will write $\pi$ for the cartesian product $(j,\chi^*)$ of $j$ and $\chi^*$.

Note that
\[\chi^* = \frac{E_2E_4E_6}{\Delta} - \frac{3}{\pi y}\frac{E_4E_6}{\Delta},\]
so it will be helpful to give the components names.  We will write $\chi =  \frac{E_2E_4E_6}{\Delta}$ and $\xi = \frac{E_4E_6}{\Delta}$, so that
\[\chi^* = \chi - \frac{3}{\pi y}\xi.\]

It is well-known of course that for quadratic numbers $\tau \in\uh$, the number $j(\tau)$ is algebraic, and referred to as a special point, or singular modulus.  It is also true, proven by Masser in \cite{Masser1975}, that $\chi^*(\tau)$ is algebraic for quadratic $\tau$, and in fact $\chi^*(\tau)\in\mathbb{Q}(j(\tau))$.  Such a point will be called a $\chi^*$-special point, while a point $\pi(\tau)=(j(\tau),\chi^*(\tau))$ is called $\pi$-special (when $\tau$ is quadratic).  

Similar special behaviour exists in positive dimensions as well: if $g\in\gl$ and $S=\{(\tau,g\tau):\tau\in\uh\}$, then the set
\[\pi(S)=\{(j(\tau),\chi^*(\tau),j(g\tau),\chi^*(g\tau)):\tau\in\uh\}\]
is contained in an irreducible 2-dimensional variety $V$ defined over $\overline{\mathbb{Q}}$.  Moreover, $V$ depends only on the determinant of $g$ when it is scaled so as to be a primitive integer matrix.  These varieties, together with the $\pi$-special points, are the building blocks of what we shall refer to as $\pi$-special varieties; for more details see \cite{Spence2016}.  In that paper we also prove the following theorem, which is the motivation for much of this paper.

\begin{theorem}[Andr\'e-Oort for $\pi$]\label{thrm:AOforPi}
	Let $V\suq\mathbb{C}^{2n}$ be an algebraic variety.  Then $V$ contains only finitely many maximal $\pi$-special varieties.
\end{theorem}
This was proven using techniques from o-minimality which are ineffective and rather difficult to make effective; it also relies heavily on a Galois bound provided by Siegel which is ineffective.  This lack of effectivity, however, is not so unusual; in general, techniques powerful enough to deal with classical Andr\'e-Oort results in full generality tend to be ineffective for similar reasons.

Several effective and explicit results of Andr\'e-Oort type do exist in the classical context, however.  For instance work of K\"uhne\cite{Kuehne2012}/Bilu-Masser-Zannier \cite{Bilu2013}, who prove:

\begin{theorem}\label{thrm:effectiveAOforJ}
	Let $V\suq\mathbb{C}^2$ be an algebraic curve defined over $\overline{\mathbb{Q}}$.  Then there are effectively computable constants $c_i=c_i(V)$ such that whenever $(j(\tau_1),j(\tau_2))\in V$ with quadratic $\tau_i$ and $d_i$ is the absolute value of the discriminant of $\tau_i$, either
	\[\max(d_1,d_2)\leq c_1\]
	or there is a primitive integer matrix $g$ of determinant at most $c_2$ such that $\tau_2 = g\tau_1$.
\end{theorem}

Besides this effective Andr\'e-Oort result for $j$, there are also a number of what we'll call ``explicit'' results; theorems which answer specific questions about the special subvarieties of particular varieties $V$.  For instance, work of Pila and Tsimerman who proved in \cite{Pila2014a} that with obvious exceptions there are only finitely many multiplicatively dependent $n$-tuples of singular moduli, or of Bilu, Luca and Masser \cite{Bilu2017}, who proved:

\begin{theorem}\label{thrm:collinearJPoints}
	Barring obvious exceptions, there are only finitely many triples 
	\[(j(\tau_1),j(\tau_2)),\qquad(j(\tau_3),j(\tau_4)),\qquad(j(\tau_5),j(\tau_6))\]
	with $\tau_i\in\uh$ all quadratic.
\end{theorem}

The goal of this paper is to investigate analogues of theorems \ref{thrm:effectiveAOforJ} and \ref{thrm:collinearJPoints} for $\chi^*$ and for $\pi$.  In Section \ref{sect:qExp} we calculate various bounds on the $q$-expansions involved and prove a weak analogue of Theorem \ref{thrm:effectiveAOforJ}, namely:
\begin{theorem}\label{thrm:EffectivePiAO}
	Let $X\suq\mathbb{C}^2$ be an algebraic curve defined over $\overline{\mathbb{Q}}$.  Then there is an effectively computable constant $c=c(X)$ such that for quadratic $\tau\in\uh$, 
	\[(j(\tau),\chi^*(\tau))\in X \implies \text{ the absolute value of the discriminant of }\tau \text{ is at most }c.\]
\end{theorem}
This is of course not a perfect analogue of \ref{thrm:effectiveAOforJ}; the ideal analogue would have two copies of $\chi^*$, rather than a $j$ and a $\chi^*$.  We discuss this very briefly in Section \ref{sect:qExp}; the summary is that we are not certain how to approach such a conjecture.

\bigskip

In Section \ref{sect:collinear} we work on a weak analogue of Theorem \ref{thrm:collinearJPoints}:
\begin{theorem}\label{thrm:collinearPiPoints}
	There are only finitely many collinear triples
	\[P_1=(j(\tau_1),\chi^*(\tau_1)),\qquad P_2=(j(\tau_2),\chi^*(\tau_2)),\qquad P_3=(j(\tau_3),\chi^*(\tau_3))\]
	with $\tau_i$ quadratic and $P_i$ pairwise distinct.
\end{theorem}
Just as \ref{thrm:EffectivePiAO} is not a perfect analogue of the classical version \ref{thrm:effectiveAOforJ}, Theorem \ref{thrm:collinearPiPoints} is not a perfect analogue to \ref{thrm:collinearJPoints}.  We would very much like to have the following:

\begin{conjecture}\label{conj:collinearChiPoints}
	There are only finitely many triples
	\[P_1=(\chi^*(\tau_1),\chi^*(\tau_2)),\qquad P_2=(\chi^*(\tau_3),\chi^*(\tau_4)),\qquad P_3=(\chi^*(\tau_5),\chi^*(\tau_6))\]
	with $\tau_i$ quadratic, such that the $P_i$ are pairwise distinct and belong to a straight line which is neither horizontal, vertical nor the diagonal $x=y$.
\end{conjecture}
This should very much be a tractable problem.  Indeed by emulating Bilu-Luca-Masser's proof of \ref{thrm:collinearJPoints}, one can get a long way towards a proof of \ref{conj:collinearChiPoints}.  Unfortunately, there remains a gap, which while apparently quite surmountable, the author has not found the time to work through.  The gap lies in the fact that the  Bilu-Luca-Masser approach relies on a particular result of Allombert, Bilu and Pizarro-Madariaga \cite[Theorem 1.2]{Allombert2015}.  Without a suitable analogue of this, two crucial lemmas from \ref{thrm:collinearJPoints} lack suitable analogues in the $\chi^*$ setting.    

One might be able to use the ``multiplicity of $q$-expansions'' contained in $\chi^*$ to circumvent the need for the missing result.  I will briefly discuss the state of my approaches towards Conjecture \ref{conj:collinearChiPoints} at the end of Section \ref{sect:collinear}; with luck, a future version of this paper might contain a complete proof.

\begin{notation}
	Throughout, we will use $\mathbb{F}$ to refer to (the closure of) the standard fundamental domain for the action of $\slz$ on $\uh$, namely:
	\[\mathbb{F}=\left\{z\in\mathbb{C}: |z| \geq 1, -\frac{1}{2}\leq\operatorname{Re} z\leq\frac{1}{2}\right\}.\]
\end{notation}

\section{Bounds on $q$-expansions and Effective Andr\'e-Oort for $\pi$}\label{sect:qExp}
The goal of most of this section is to carry out the explicit $q$-expansion calculations which form much of the basis for the effective results to come.  We will at the end of the section use these to prove our effective Andr\'e-Oort result \ref{thrm:EffectivePiAO} for $\pi$, which is the easiest of the results in this paper.

The bounds we need are on the tails of the relevant $q$-expansions; we wish to show that the first term in each $q$-expansion is the main contributor to the total value of the expansion.  The $q$-expansions of $\chi$, $\xi$ and $j$ all begin with $q^{-1}$, so we will write:
\[j = q^{-1}+\widehat{j},\qquad \chi = q^{-1}+\widehat{\chi},\qquad \xi = q^{-1}+\widehat{\xi}.\]
We wish to estimate $|\widehat{\chi}|$ and $|\widehat{\xi}|$, aiming for some analogue of known facts about $\widehat{j}$; it was proven by Bilu-Masser-Zannier in \cite{Bilu2013} that $|\widehat{j}|\leq 2079$ for all $\tau \in \mathbb{F}$.  We'll be getting analogues of this fact for $\widehat{\chi}$ and $\widehat{\xi}$, working from first principles starting with the known $q$-expansions of Eisenstein series.  As with $\widehat{j}$,  we'll be able to get much better bounds when $\Imm\tau \geq 2$, which will be useful later, so we will distinguish cases based on the size of $\Imm\tau$.

\bigskip

For all $n\geq 3$, Robin \cite{Robin1984} proved that
\[\sigma(n) < e^\gamma n\log\log n + \frac{0.6483n}{\log\log n}. \]
For $n\geq 4$ we can rewrite this as the more manageable
\[\sigma(n) < 8 n\log\log n.\]
Even better, for $n\geq 6$ we get
\[\sigma(n)< 4n\log\log n.\]
It trivially follows that for $n\geq 6$:
\[\sigma_3(n) < 64n^3(\log\log n)^3,\]
\[\sigma_5(n) < 1024n^5(\log\log n)^5,\]
and one can check by hand that in fact the above two inequalities hold for $n=4$ and $5$ as well.

Using these, we can get bounds on the tails of the $q$-expansions of the Eisenstein series $E_2$, $E_4$, $E_6$.  First note that for $\tau\in\mathbb{F}$,
\[|e^{2\pi i \tau}| = |e^{-2\pi \Imm\tau}| \leq e^{-\pi\sqrt{3}} < 0.005. \]
Using this we see that:
\begin{align*}|(E_2(\tau)-1)/q|&\leq 24  \left(1 + 0.015 + 0.0001 + 200\sum_{n\geq 4} 8n\log\log n 0.005^n\right)\\&< 24\left(1.016+1600\sum_{n\geq 4} n 0.0055^n\right) \\&<  24 \times 1.017\\&<25,
\\|(E_4(\tau)-1)/q|&\leq 240  \left(1 + 0.045 + 0.001  + 200\sum_{n\geq 4} 64n^3(\log\log n)^3 0.005^n\right)\\&< 240\left(1.046+12800\sum_{n\geq 4} n^3 0.0067^n\right) \\&<  240 \times 1.048\\&<252,
\\|(E_6(\tau)-1)/q|&\leq 504 \left(1 + 0.165 + 0.007 + 200\sum_{n\geq 4} 1024n^5(\log\log n)^5 0.005^n\right)\\&< 504\left(1.172+204800\sum_{n\geq 4} n^5 0.0081^n\right) \\&<  504 \times 2.172\\&=1095.
\end{align*}
In each case we are using standard methods to evaluate the infinite sum and also using the fact that, for all $n$, $1.1^n > \log\log n$. 

Now we can begin work on $\widehat{\chi}$ and $\widehat{\xi}$, aiming first to achieve some strong bounds holding only for certain $\tau$.

\begin{proposition}\label{propn:strongerChiBounds}
	For $\Imm\tau\geq 2$, we have
	\[|\widehat{j}|\leq 1193,\qquad|\widehat{\chi}|\leq 4808,\qquad\text{and}\qquad|\widehat{\xi}|\leq 4782.\]
	\begin{proof}
		The fact for $\widehat{j}$ is due to K\"uhne, who proved it in \cite{Kuehne2012}.  The claims for $\chi$ and $\xi$ will require a little work.
		
		We first need to find a suitable lower bound on $\Delta$.  More specifically, we need to find an effective constant $c$ such that
		\[\left|\frac{\Delta - q}{q^2}\right| < c\]
		for all $\tau\in\mathbb{F}$.   My somewhat crude method uses the fact that $|\widehat{j}|\leq 1193$, from which it follows immediately that
		\[|1-jq|<1193|q|,\]
		which, for $\Imm\tau\geq 2$, is bounded above by $0.01$, whence 
		\begin{equation}\label{eqn:jqBound}|(jq)^{-1}|<(0.99)^{-1}<1.011.\end{equation}
		
		We also know that for $\tau\in\mathbb{F}$,
		\begin{multline}\label{eqn:e4cubedBound}\left|\frac{E_4^3-1}{q}\right|=\left|\frac{\left(1+q\frac{E_4-1}{q}\right)^3-1}{q}\right|=\left|3\frac{E_4-1}{q}+3q\left(\frac{E_4-1}{q}\right)^2+q^2\left(\frac{E_4-1}{q}\right)^3\right|\\<756+953+401=2110.\end{multline}
		
		We can write
		\[\frac{\Delta - q}{q^2} = \frac{\frac{E_4^3}{j}-q}{q^2} = \frac{1}{jq}\left(\frac{E_4^3-jq}{q}\right),\]
		whence
		\[\left|\frac{\Delta-q}{q^2}\right| \leq \left|\frac{1}{jq}\right|\left|\frac{E_4^3-1}{q}\right|+\left|\frac{1}{jq}\right||j-q^{-1}|<\left|\frac{1}{jq}\right|\left(\left|\frac{E_4^3-1}{q}\right|+1193\right).\]
		Combining this with (\ref{eqn:jqBound}) and (\ref{eqn:e4cubedBound}) yields
		\begin{equation}\label{eqn:DiscBound}\left|\frac{\Delta - q}{q^2}\right|< 1.011\times(2110+1193) < 3340.\end{equation}
		
		\bigskip
		
		By direct calculation we can see that
		\begin{multline}\label{eqn:firstChiBound}|\widehat{\chi}|=|\chi-q^{-1}|=\left|\frac{\left(1+q\frac{E_2-1}{q}\right)\left(1+q\frac{E_4-1}{q}\right)\left(1+q\frac{E_6-1}{q}\right)-\left(1+q\frac{\Delta-q}{q^2}\right)}{q\left(1+q\frac{\Delta-q}{q^2}\right)}\right|
		\\\leq\left|\frac{1}{1+q\frac{\Delta-q}{q^2}}\right|\times\left(\left|\frac{\Delta - q}{q^2}\right|+\left|\frac{E_2-1}{q}\right|+\left|\frac{E_4-1}{q}\right|+\left|\frac{E_6-1}{q}\right|\right.+\\\left.\left|q\frac{E_2-1}{q}\frac{E_4-1}{q}\right|+\left|q\frac{E_2-1}{q}\frac{E_6-1}{q}\right|+\left|q\frac{E_4-1}{q}\frac{E_6-1}{q}\right|+\left|q^2\frac{E_2-1}{q}\frac{E_4-1}{q}\frac{E_6-1}{q}\right|\right),\end{multline}
		which, using (\ref{eqn:DiscBound}) is bounded above by
		\[1.02\times(3340+25+252+1095+0.1+0.1+1+0.1)<4808.\]
		
		Similarly:
		\begin{align*}|\widehat{\xi}|=|\xi-q^{-1}|&\leq\left|\frac{1}{1+q\frac{\Delta-q}{q^2}}\right|\times\left(\left|\frac{\Delta - q}{q^2}\right|+\left|\frac{E_4-1}{q}\right|+\left|\frac{E_6-1}{q}\right|+\left|q\frac{E_4-1}{q}\frac{E_6-1}{q}\right|\right) \\&<1.02\times(3340+252+1095+1)<4782.\end{align*}

	\end{proof}
\end{proposition}

For the remainder of the $\tau\in\mathbb{F}$, the bounds we can get are not as good.

\begin{proposition}\label{propn:weakerChiBounds}
	For all $\tau\in\mathbb{F}$,
	\[\left|\widehat{\chi}\right| < 39960\qquad\text{and}\qquad|\widehat{\xi}| < 39032.\]
	\begin{proof}
		This divides into two steps: first we'll get such bounds for $\Imm\tau \geq 1.5$, then for the remainder of the region.
		
		For $\Imm\tau\geq 1.5$, we proceed exactly as above.  We have $|1-jq|<1193|q|$, which for $\Imm\tau\geq 1.5$ is bounded above by 0.1, so as in (\ref{eqn:DiscBound}):
		\[\left|\frac{\Delta-q}{q^2}\right|< 1.12\times(2110+1193)<3700.\]
		It follows as for (\ref{eqn:firstChiBound}) that when $\Imm\tau\geq 1.5$,
		\[|\widehat{\chi}| \leq 1.43\times(3700+25+252+1095+1+3+28+1)<7299\]
		and similarly
		\[|\widehat{\xi}| < 7258.\]

		\bigskip

		For the remainder of the $\tau\in\mathbb{F}$ (ie. those with $\frac{\sqrt{3}}{2}\leq\Imm\tau<1.5$) we use a different technique.  Recalling that $\Delta = q\prod_{n=1}^{\infty}(1-q^n)^{24}$ and noting that in the desired region $|q^{-1}|\leq 12392$ and $|q|<0.005$, we see:
		\begin{align*}
		\left|\frac{\Delta - q}{q^2}\right| = \left|\left(q^{-1}- 1\right)(1-q)^{23}\prod_{n=2}^{\infty}(1-q^n)^{24} - q^{-1}\right|\leq \left|12393\times 0.9\right|+12392 < 23546.
		\end{align*}
		Also we have
		\begin{align*}\left|\frac{1}{1+q\frac{\Delta - q}{q^2}}\right| = \left|\frac{q}{\Delta}\right|=\left|\frac{1}{\prod_{n=1}^{\infty}(1-q^n)^{24}}\right|&<\left|\frac{1}{\prod_{n=1}^{100}(1-0.005^n)\prod_{n=101}^{\infty}(1-n^{-2})}\right|^{24}\\&<\left(\frac{1}{0.994\times\frac{101}{102}}\right)^{24}<1.5.\end{align*}
		In much the same way as for (\ref{eqn:firstChiBound}), we then get
		\[\left|\widehat{\chi}\right| < 1.5\times (23546 + 25+252+1095+32+137+1380+173) = 39960\]
		and
		\[|\widehat{\xi}| < 39032,\]
		as required.		
	\end{proof}
\end{proposition}

We'll use all of the above propositions in the calculations to come, but in most cases it will make more sense to use the better bounds known for $\Imm\tau\geq 2$, ie. Proposition \ref{propn:strongerChiBounds}.  The main purpose for getting Proposition \ref{propn:weakerChiBounds}  was to yield the following lemmas. 

As with many of the calculations in this paper, the first of these lemmas is taken essentially verbatim from Lemma 5.1 of \cite{Bilu2017}, in light of the calculations above.

\begin{lemma}\label{lma:QuadraticSizeComparison}
	Let $x=\chi^*(\tau)$ be a $\chi^*$-special point and $x'=\chi^*(\tau')$ the principal $\chi^*$-special point of the same discriminant.  Assume (without any loss) that $\tau$ and $\tau'$ are each in $\mathbb{F}$.  Then either $\tau=\tau'$ or $|x'|>|x|+5595$.
	\begin{proof}
		Let $D$ be the common discriminant of $x$ and $x'$.  We may assume that $|D|\geq 15$, otherwise $h(D)=1$ and there is nothing to prove.  We'll assume that $\tau\ne \tau'$.
		
		Since $\tau'$ is principal and $\tau$ is non-principal, it follows that
		\[\Imm\tau' = \frac{\sqrt{|D|}}{2}\qquad\text{and}\qquad\Imm\tau \leq \frac{\sqrt{|D|}}{4}.\]
		Therefore
		\begin{align*}|x'|=|\chi^*(\tau')|&=\left|q^{-1}\left(1-\frac{3}{\pi \Imm\tau'}\right) + \widehat{\chi}-\frac{3}{\pi \Imm\tau'}\widehat{\xi}\right|\\&\geq|q^{-1}|\left(1-\frac{6}{\pi\sqrt{15}}\right)-4808 - \frac{6}{\pi\sqrt{15}}\times 4782\qquad\text{using Proposition \ref{propn:strongerChiBounds}.}\\&\geq e^{\pi\sqrt{|D|}}\times 0.5-7167\end{align*}
		
		On the other hand 
		\begin{align*}
			|x|=|\chi^*(\tau)|&\leq |q^{-1}|\left|1-\frac{3}{\pi\Imm\tau}\right|+39960 + \frac{6}{\pi\sqrt{3}}\times 39032\qquad\text{using Proposition \ref{propn:weakerChiBounds}.}\\&\leq |q^{-1}|+39960 + 43039\qquad\text{since }\Imm\tau\geq \sqrt{3}/2\text{ and }|1-6/\pi\sqrt{3}|<1.\\&\leq e^{\pi\sqrt{|D|}/2} + 82999.
		\end{align*}
		So 
		\[|x'|-|x| \geq e^{\pi\sqrt{|D|}}\times 0.5- e^{\pi\sqrt{|D|}/2} - 82999-7167 \geq 0.5e^{\pi\sqrt{15}}-e^{\pi\sqrt{15}/2}-90166 > 5595, \]
		as required.
	\end{proof}
\end{lemma}
The above will be useful for the next section, but it also allows us to resolve an old question about the degree of $\chi^*(\tau)$, which was discussed in \cite{Spence2017Ext} and in \cite{Spence2016}.
\begin{corollary}\label{cor:FieldEquality}
	For quadratic $\tau$, $\mathbb{Q}(j(\tau))=\mathbb{Q}(\chi^*(\tau))$.
	\begin{proof}
		For a given discriminant $D$, let $S_D$ be the set of $\tau\in\mathbb{F}$ having discriminant $D$.  The class polynomial 
		\[H_j(X) = \prod_{\tau\in S_D}(X-j(\tau))\in\mathbb{Q}(X)\]
		is known to be irreducible over $\mathbb{Q}$.
		
		It is a fact due to Masser that for quadratic $\tau$, $\mathbb{Q}(\chi^*(\tau))\suq\mathbb{Q}(j(\tau))$.  Proposition 5.2 of \cite{Spence2016} uses work of Masser to show further that for quadratic $\tau$, if $\theta$ is a field automorphism acting on $\mathbb{Q}(j(\tau))$ so that $\theta(j(\tau))=j(\tau')$, then also $\theta(\chi^*(\tau))=\chi^*(\tau')$.
		
		From this and the irreducibility of $H_j(X)$, it follows that the class polynomial
		\[H_{\chi^*}(X)=\prod_{\tau\in S_D}(X-\chi^*(\tau))\] 
		takes the form $p(X)^k$ for some polynomial $p$ irreducible over $\mathbb{Q}$.
		
		Now let $\tau'\in S_D$ be principal.  If $k>1$, then it follows that there is a non-principal $\tau\in S_D$ such that $\chi^*(\tau)=\chi^*(\tau')$, which contradicts Lemma \ref{lma:QuadraticSizeComparison}.
		
		So $H_{\chi^*}(X)$ is irreducible, whence $[\mathbb{Q}(\chi^*(\tau)):\mathbb{Q}]=[\mathbb{Q}(j(\tau)):\mathbb{Q}]$, so that indeed $\mathbb{Q}(\chi^*(\tau))=\mathbb{Q}(j(\tau))$.
	\end{proof}
\end{corollary}

We can also now prove our desired effective Andr\'e-Oort result for $\pi$, namely Theorem \ref{thrm:EffectivePiAO}
	\begin{proof}[Proof of Theorem \ref{thrm:EffectivePiAO}]
		For a number field $K$, let $p\in K[X,Y]$.  Suppose that $p(j(\tau),\chi^{*}(\tau))$ vanishes for some quadratic $\tau$ of discriminant $-D$.  For every quadratic $\tau'$ also having discriminant $-D$, there is a Galois automorphism (over $\mathbb{Q}$) sending $j(\tau)$ to $j(\tau')$.  Call it $\theta$.  By Proposition 5.2 of \cite{Spence2016}, we know that also $\theta(\chi^*(\tau)) = \chi^*(\tau')$, so that $\theta(p)(j(\tau'),\chi^*(\tau'))$ vanishes.  Hence, by considering $\theta(p)$ rather than $p$, we may assume that $\tau = \frac{D+\sqrt{-D}}{2}$. 
		
		We will write $h(p)$ for the maximum of the absolute logarithmic heights of the coefficients of $p$, as defined in \cite{Bombieri2006}.  Since the absolute logarithmic height is Galois invariant, we may carry out the reduction described in the previous paragraph without affecting this height.   It follows from an inequality of Liouville \cite{Zannier2009} that, if $\alpha$ is a coefficient occurring in $p$,
		\begin{equation}\label{eqn:LiouvilleHeightBound}-[K:\mathbb{Q}]h(p)\leq\log|\alpha|\leq[K:\mathbb{Q}]h(p),\end{equation}
		an inequality which we will use on a number of occasions below.
		
		\bigskip

		We can write 
		\[p(j(\tau),\chi^*(\tau)) = p\left(q^{-1}+\widehat{j}, q^{-1}+\widehat{\chi}-\frac{3}{\pi \Imm\tau}\left(q^{-1}+\widehat{\xi}\right)\right)\]
		and we get a Laurent series in $q$, whose coefficients are polynomials in $3/\Imm\tau$.  Note that there is no cancellation among the $\frac{3}{\pi\Imm\tau}$ terms.  So the leading $\frac{3}{\pi\Imm\tau}$ term of the leading term in the $q$-series comes directly from the leading $Y$ term of the leading $X$ term of $p(X,Y)$.
		
		The Laurent series therefore takes the form:
		\[q^{-\deg p}\left(A\left(\frac{3}{\pi\Imm\tau}\right)^k + p_1\left(\frac{3}{\pi\Imm\tau}\right)\right)+p_2\left(q^{-1}, \widehat{j}, \widehat{\chi},\widehat{\xi},\frac{3}{\pi\Imm\tau}\right),\]
		where:
		\begin{itemize}
			\item $A$ is one of the coefficients of $p$, so in particular \begin{equation}\label{eqn:LowerBoundOnA}|A|\geq e^{-[K:\mathbb{Q}]h(p)}.\end{equation}
			\item The degree of $p_1$ is less than $k$.
			\item The degree of $p_2(X,A,B,C,D)$ in $X$ is less then $\deg p$.
			\item Using (\ref{eqn:LiouvilleHeightBound}), the absolute values of the coefficients of the $p_i$ are bounded above by a constant which can easily be computed in terms of $h(p)$, $[K:\mathbb{Q}]$ and $\deg{p}$.
		\end{itemize}
		So we can rewrite $p(j,\chi^*)=0$ as
		\[A = -\left(\frac{3}{\pi\Imm\tau}\right)^{-k}p_1\left(\frac{3}{\pi\Imm\tau}\right)-q^{\deg p}\left(\frac{3}{\pi\Imm\tau}\right)^{-k}p_2\left(q^{-1}, \widehat{j}, \widehat{\chi},\widehat{\xi},\frac{3}{\pi\Imm\tau}\right)\]
		Provided that $\Imm\tau \geq 2$, the absolute value of the right hand side is bounded above by
		\[\left(\frac{3}{\pi\Imm\tau}\right)\cdot c_1(h(p),[K:\mathbb{Q}], \deg p) + q\cdot\left(\frac{\pi\Imm\tau}{3}\right)^{k}c_2(h(p),[K:\mathbb{Q}],\deg p, 1193, 4782, 4808),\]
		for some easily computed constants $c_1$, $c_2$.  We noted in (\ref{eqn:LowerBoundOnA}), though, that $|A|\geq e^{-[K:\mathbb{Q}]h(p)}$.  Writing $H(p)=e^{[K:\mathbb{Q}]h(p)}$, these inequalities are inconsistent if we have both:
		\[\Imm\tau > \frac{6\operatorname{H}(p)c_1}{\pi}\]
		and 
		\[e^{-2\pi\Imm\tau}\left(\frac{\pi\Imm\tau}{3}\right)^{k}c_2 < \frac{1}{2\operatorname{H}(p)}.\]
		Since $x^k<k!e^x$ for all $x$, this final condition holds provided that
		\[e^{(1-2\pi)\Imm\tau}<\frac{3^k}{2\pi^kk!\operatorname{H}(p)c_2},\]
		whence it suffices to have
		\[\Imm\tau > |1-2\pi|^{-1}\log\left(\frac{2\pi^kk!\operatorname{H}(p)c_2}{3^k}\right).\] 		
		
		So if $p(j(\tau),\chi^*(\tau))$ vanishes, it follows that the above inequalities cannot hold.  In other words, the discriminant $-D$ of $\tau$ must satisfy
		\[D\leq 4\max\left(\frac{6\operatorname{H}(p)c_1}{\pi}, |1-2\pi|^{-1}\log\left(\frac{2\pi^kk!\operatorname{H}(p)c_2}{3^k}\right) \right)^2.  \]
	\end{proof}
The observant reader will note a discrepancy between the statement of Theorem \ref{thrm:EffectivePiAO} and that of the motivating theorem for $j$, \ref{thrm:effectiveAOforJ}.  Specifically, note that there are more degrees of freedom present in Theorem \ref{thrm:effectiveAOforJ} than in \ref{thrm:EffectivePiAO}.  A more direct analogue might be something like the following:

\begin{conjecture}\label{conj:effectiveChiAO}
	Let $V\suq\mathbb{C}^2$ be an algebraic curve defined over $\mathbb{Q}$.  Then there are effectively computable constants $c_i=c_i(V)$ such that whenever $(\chi^*(\tau_1),\chi^*(\tau_2))\in V$ with quadratic $\tau_i$ and $d_i$ is the absolute value of the discriminant of $\tau_i$, either
	\[\max(d_1,d_2)\leq c_1\]
	or there is a primitive integer matrix $g$ of determinant at most $c_2$ such that $\tau_2 = g\tau_1$.
\end{conjecture}

Unfortunately, the techniques used by K\"uhne/Bilu-Masser-Zannier to prove this for $j$ relied heavily on effective estimates by Baker on logarithms in algebraic numbers.  The presence of the transcendental number $\pi$ in the expression $\chi^*=\chi-\frac{3}{\pi y}\xi$ (as well as in the $q$-expansions) interferes with this method and prevents us from carrying out the K\"uhne/Bilu-Masser-Zannier approach in full.  In proving \ref{thrm:EffectivePiAO} we have essentially carried out the easy half of the K\"uhne/Bilu-Masser-Zannier approach; the part which need not appeal to Baker's theorem.  Conjecture \ref{conj:effectiveChiAO}, while not strictly stronger than Theorem \ref{thrm:EffectivePiAO}, certainly seems more difficult to approach in the absence of a suitable Baker-like result.

\section{Collinear Special Points}\label{sect:collinear}
A triple of points
\[P_1=(x_1, y_1),\qquad P_2=(x_2, y_2),\qquad P_3=(x_3, y_3)\]
is collinear if and only if the determinant
\[\begin{vmatrix}1&1&1\\x_1&x_2&x_3\\y_1&y_2&y_3\end{vmatrix}\]
vanishes.

In order to approach Theorem \ref{thrm:collinearPiPoints}, we use the above to define a variety $V\suq\mathbb{C}^6$.  Sets of collinear special points will therefore correspond to special points in $V$, and by Theorem \ref{thrm:AOforPi}, $V$ will contain only finitely many such points unless it contains a positive-dimensional subvariety.

This approach is very much the same as that used by \cite{Bilu2017} to prove similar results for $j$, and indeed much of the work will be left to that paper rather than replicate the details exactly.  In the presence of the estimates from the previous section, we are able to jump straight to the proof of Theorem \ref{thrm:collinearPiPoints}, with a minimum of additional setup.  We take the following notation straight from \cite{Bilu2017}.

\begin{definition}
	A function from $\uh$ to $\mathbb{C}$ is called a $j$-map if either it takes the form
	\[j_g = j\circ g: \tau\mapsto j(g\tau)\]
	for some $g\in\gl$ or is a constant map $j_{\tau_0}$ which sends everything to some $j(\tau_0)$, with $\tau_0$ quadratic.  
	
	Similarly, a function is called a $\chi$-map if it takes the form $\chi_g=\chi^*\circ g$ or is constant and special.  Note that in this definition, for a given $j$- or $\chi$-map $j_g$ (resp. $\chi_g$), one can always without loss choose $g\in\gl$ to take the form
	\[\begin{pmatrix}a&b\\0&d\end{pmatrix}\]
	with $b<d$.  In this case the number $a/d$ is called the \emph{level} of the map and the root of unity $e^{2\pi ib/d}$ is called the \emph{twist}.  Note that a nonconstant $j$- or $\chi$-map is defined by its level and twist.  The level of a constant $j$- or $\chi$-map is defined to be 0 and the twist undefined.
	
	A pair $(F,G)$ consisting of a $j$-map $F$ and $\chi$-map $G$ is \emph{consistent} if $F$ and $G$ have the same level and twist and, if they are constant maps, they come from the same element of $\uh$.  If $(F,G)$ is a consistent pair we will often refer to its level and/or twist, in the obvious way.
\end{definition}

\begin{proof}[Proof of Theorem \ref{thrm:collinearPiPoints}]
	Suppose there were infinitely many pairwise distinct triples 
	\[P_1=(j(\tau_1),\chi^*(\tau_1)),\qquad P_2=(j(\tau_2),\chi^*(\tau_2)),\qquad P_3=(j(\tau_3),\chi^*(\tau_3))\]
	with $\tau_i$ quadratic.  Then by Theorem \ref{thrm:AOforPi}, the variety $V$ contains a positive-dimensional special-subvariety, excluding those defined by equations of the form $x_i = x_j$, $y_i = y_j$, $i\ne j$.  
	
	This would imply the existence of 3 pairwise distinct triples
	\[F_1=(j_1,\chi_1),\qquad F_2=(j_2,\chi_2),\qquad F_3=(j_3,\chi_3)\]
	 where each $j_i$ is a $j$-map, each $\chi_i$ is a $\chi$-map, each $F_i$ is consistent, at least one of the maps is nonconstant and
	 \begin{equation}\label{eqn:PiDeterminant}\begin{vmatrix}1&1&1\\j_1&j_2&j_3\\\chi_1&\chi_2&\chi_3\end{vmatrix}=0\end{equation}
	 identically.
	 
	 We'll use Lemma 7.2 from \cite{Bilu2017}, which tells us that for any three distinct $j$-maps, not all constant, by composing with an element of $\gl$, we can ensure that one of them has strictly higher level than the other two.  We can therefore without loss of generality assume that the level of $F_1$ is greater than the levels of the other two pairs.  
	 
	 Note that if $F_2$ and $F_3$ are both constant, then since they are distinct, (\ref{eqn:PiDeterminant}) induces a nontrivial relation between $j_1$ and $\chi_1$, which is impossible: $j$ and $\chi^*$ are algebraically independent.  So at most one of $F_2$ and $F_3$ can be constant.
	 
	 \bigskip\noindent
	 \textbf{Case 1: Neither $F_2$ nor $F_3$ is constant.}
	 
	 Let us write $r_i$ for the level of $F_i$ and $\eta_i$ for its twist. 
	 
	 Since, as proven in Proposition \ref{propn:weakerChiBounds}, the tails of all the relevant $q$-expansions are bounded within $\mathbb{F}$, the only way (\ref{eqn:PiDeterminant}) can hold identically is if
	 \begin{equation}\label{eqn:PiDeterminantTwo}\begin{vmatrix}1&1&1\\\eta_1q^{-r_1}&\eta_2q^{-r_2}&\eta_3q^{-r_3}\\\eta_1q^{-r_1}-\frac{3}{\pi r_1 y}\eta_1q^{-r_1}&\eta_2q^{-r_2}-\frac{3}{\pi r_2 y}\eta_2q^{-r_2}&\eta_3q^{-r_3}-\frac{3}{\pi r_3 y}\eta_3q^{-r_3}\end{vmatrix}=0,\end{equation}
	 since the above expression accounts for all of the dominant terms.  
	 
	 We know already that $r_1$ is greater than $r_2$ and $r_3$.  So if $r_2>r_3$, then as $\Imm\tau$ grows, the dominant term of (\ref{eqn:PiDeterminantTwo}) is 
	 \[\frac{3}{\pi y}\eta_1\eta_2\left(r_1^{-1}-r_2^{-1}\right)q^{-r_1-r_2}.\]
	 Since $r_1>r_2$, this term is nonzero and so the determinant above cannot vanish identically.  Contradiction.  Symmetrically we cannot have $r_3>r_2$.
	 
	 We therefore have $r_2=r_3$, in which case the dominant term of (\ref{eqn:PiDeterminantTwo}) is 
	 \[\frac{3}{\pi y}\eta_1(\eta_2-\eta_3)\left(r_1^{-1}-r_2^{-1}\right)q^{-r_1-r_2},\]
	 which can only vanish if $\eta_2=\eta_3$, since $r_1>r_2$.  Since the levels and twists of $F_2$ and $F_3$ now match, we have $F_2=F_3$, which we assumed was not the case.  Contradiction.
	 
	 \bigskip\noindent
	 \textbf{Case 2: Without loss of generality, $F_3$ is constant.}
	 
	 Write $F_3=(a,b)$.  Exactly as above, since the tails of the relevant $q$-expansions are bounded, (\ref{eqn:PiDeterminant}) can only vanish identically if 
	 \[\begin{vmatrix}1&1&1\\\eta_1q^{-r_1}&\eta_2q^{-r_2}&a\\\eta_1q^{-r_1}-\frac{3}{\pi r_1 y}\eta_1q^{-r_1}&\eta_2q^{-r_2}-\frac{3}{\pi r_2 y}\eta_2q^{-r_2}&b\end{vmatrix}=0,\]
	 since this contains all the potentially dominant terms.  In fact the dominant term here is
	 \[\frac{3}{\pi y}\eta_1\eta_2\left(r_1^{-1}-r_2^{-1}\right)q^{-r_1-r_2}\]
	 which can't vanish since $r_1>r_2$. 	 
\end{proof}

\subsection{Collinearity for $\chi^*$ Alone}
In this final part, I will briefly discuss my initial attempts to prove Conjecture \ref{conj:collinearChiPoints}.  I am convinced that the conjecture should be attainable with a relative minimum of work - no new ideas being required - but have not had the time to work this all the way through.  The approach I have in mind is exactly the same as that used in \cite{Bilu2017}; I will describe how that goes in this context and where there are remaining gaps.

Suppose we had infinitely many collinear triples
\[(\chi^*(\tau_1),\chi^*(\sigma_1)), (\chi^*(\tau_2),\chi^*(\sigma_2)),\text{ and } (\chi^*(\tau_3),\chi^*(\sigma_3)),\]
with $\tau_i$, $\sigma_j$ quadratic and the 3 pairs being distinct.  We exclude the obvious cases where the line in question is the diagonal $X=Y$ or is a horizontal or vertical line.

In light of Theorem \ref{thrm:EffectivePiAO} (which easily implies a version for $\chi^*$ alone by projection of coordinates), we get a collection of six $\chi$-maps $f_1$, $f_2$, $f_3$, $g_1$, $g_2$, $g_3$, not all constant, with the property that
\begin{equation}\label{eqn:ChiDeterminant}\begin{vmatrix}1&1&1\\f_1&f_2&f_3\\g_1&g_2&g_3\end{vmatrix}=0,\end{equation}
and moreover none of the following hold:
\begin{enumerate}
	\item $f_1=f_2=f_3$, (to exclude vertical lines)
	\item $g_1=g_2=g_3$, (to exclude horizontal lines)
	\item $f_i=f_j$ and $g_i=g_j$ for any $i\ne j$ (the pairs $(f_i, g_i)$ are distinct), nor
	\item $f_i=g_i$ for all $i$ (to exclude the diagonal).	
\end{enumerate}
Following \cite{Bilu2017}, the idea is to show that if 1, 2 and 3 all fail to hold then in fact 4 must hold, making this set-up impossible.

\bigskip
Exactly as in \cite{Bilu2017}, we can write $m_i$, $n_i$ for the level of $f_i$, $g_i$ respectively, and $\epsilon_i$, $\eta_i$ for their twists (where they are nonconstant).  Under the assumption that 1, 2 and 4 all fail to hold, the work lies in proving that $m_i=n_i$ and $\epsilon_i=\eta_i$ for all $i$.  This involves conditioning on various inequalities between the $m_i$ and $n_i$ and, for each of the various cases, perform some $q$-expansion calculations.

The calculation differs depending on which, if any, of the $\chi$-maps are constant.  As an example, we will demonstrate in the case where just $f_3$ and $g_3$ only are constant.  So (\ref{eqn:ChiDeterminant}) becomes
\[\begin{vmatrix}1&1&1\\(\epsilon_1q^{-m_1} - \dots) - \frac{3}{\pi m_1y}(\epsilon_1q^{-m_1} - \dots)&(\epsilon_2q^{-m_2} - \dots) - \frac{3}{\pi m_2y}(\epsilon_2q^{-m_2} - \dots)&a\\(\eta_1q^{-n_1} - \dots) - \frac{3}{\pi n_1y}(\eta_1q^{-n_1} - \dots)&(\eta_2q^{-n_2} - \dots) - \frac{3}{\pi n_2y}(\eta_2q^{-n_2} - \dots)&b\end{vmatrix}=0,\]
with $a$ and $b$ being $\chi^*$-special points.  

Note that by comparison of growth rates, in order for the above to hold, the $q$-expansions corresponding to the holomorphic part $\chi$ and the nonholomorphic part $\frac{3}{\pi y}\xi$ must vanish separately, that is:
\begin{equation}\label{eqn:ChiqExpMatrix}\begin{vmatrix}1&1&1\\\epsilon_1q^{-m_1} - 264 - 135602\epsilon_1q^{m_1}&\epsilon_2q^{-m_2} - 264 - 135602\epsilon_2q^{m_2}&a\\\eta_1q^{-n_1}  - 264 - 135602\eta_1q^{n_1}&\eta_2q^{-n_2}  - 264 - 135602\eta_2q^{n_1}&b\end{vmatrix}=0\end{equation}
and
\begin{equation}\label{eqn:XiqExpMatrix}\begin{vmatrix}1&1&1\\-m_1^{-1}(\epsilon_1q^{-m_1} - 240 - 8511777\epsilon_1q^{m_1})&-m_2^{-1}(\epsilon_2q^{-m_2} - 240 - 8511777\epsilon_2q^{m_2})&a\\-n_1^{-1}(\eta_1q^{-n_1} - 240 - 8511777\eta_1q^{n_1})&-n_2^{-1}(\eta_2q^{-n_2} - 240 - 8511777\eta_2q^{n_2})&b\end{vmatrix}=0.\end{equation}
Here we're calculating the first few terms of the $q$-expansions of $\chi$ and $\xi$ using the known $q$-expansions of the Eisenstein series $E_2$, $E_4$ and $E_6$ from Section \ref{sect:intro}.

\bigskip

Using just (\ref{eqn:ChiqExpMatrix}) - and equivalents for the cases where other $\chi$-maps are constant - we can replicate much of the calculation from \cite{Bilu2017}.  The primary relevant sections of \cite{Bilu2017} are Sections 8 onwards, where the case-analysis is carried out.  In \cite{Bilu2017}, the analogy of (\ref{eqn:ChiqExpMatrix}) simply has the leading terms of the $q$-expansion of $j$ ($q^{-1} + 744 + 196884 q$) rather than those for $\chi$.  There is no analogy for equation (\ref{eqn:XiqExpMatrix}), since $j$ has no nonholomorphic part, so for now we ignore (\ref{eqn:XiqExpMatrix})

In light of equation (\ref{eqn:ChiqExpMatrix}), then, much of the calculation in \cite[Sections 8 onwards]{Bilu2017}, goes through without a hitch, replacing occurrences of the numbers 744 and 196884 with -264 and -135602 respectively.  On a number of occasions, however, we need to appeal to appropriate analogues lemmas from earlier in \cite{Bilu2017}.

\bigskip

The lemmas from \cite{Bilu2017} in question are: 4.1, 4.2, 5.1 to 5.9, and 7.3.  

Lemmas 4.1 and 4.2 just concern roots of unity, so still hold here.  Lemma 5.1 is just the $j$-analogue of this paper's Lemma \ref{lma:QuadraticSizeComparison} and using the estimates from Section \ref{sect:qExp}, a suitable analogue of Lemma 5.3 can be attained.  An analogue of Lemma 5.4 is easy in light of this paper's Corollary \ref{cor:FieldEquality}.

Lemmas 5.6 and 5.7 claim that certain numbers like $744\pm 196884$ and $744\pm 196884\theta$ (for $\theta$ a root of unity) are not singular moduli.  The equivalent statements for $\chi^*$-special points (with 744 and 196884 replaced by -264 and -135602) are easily proven in light of the following table of integral $\chi^*$-special points, together with Corollary \ref{cor:FieldEquality}, which implies among other things that these are the only integral $\chi^*$-special points.
\begin{center}
	\begin{tabular}{c|c c c c c c c c c}
		Discriminant & -3 & -4 & -7 & -8 & -11 & -12 & -16 & -19 & -27 \\
		$\chi^*$ & 0 & 0 & -1215 & 2240 & -14336 & 23760 & 149688 & -497664 & -7772160 \\
		\hline
	\end{tabular}
	\begin{tabular}{c|c c c c}
		Discriminant &-28 &-43 &-67&-163\\
		$\chi^*$ &10596015&-627056640 &-112852776960 &-223263987730882560
	\end{tabular}	
\end{center}
This table is derived easily from Masser's list of special points calculated in \cite{Masser1975}.

Lemmas 5.8 and 5.9 gives restrictions for when a singular modulus can be an integral linear combination of certain roots of unity.  A suitable analogue for $\chi^*$-special points can be attained in light of the estimates from Section \ref{sect:qExp}, using the same strategies as employed in \cite{Bilu2017} to get Lemmas 5.8 and 5.9.

A $\chi^*$-analogue of Lemma 7.3, which gives restrictions on when two $j$-maps $f$ and $g$ can satisfy $af+bg+c=0$, is very easy in this setting.  If any non-obvious such linear relation exists among $\chi$-maps, then we can combine it with the known algebraic relations (modular polynomials) between $j$- and $\chi$-maps, counting conditions to show that some consistent $(j,\chi)$-map pair $(f,g)$ is a solution to some bivariate polynomial, which is impossible since $j$ and $\chi^*$ are algebraically independent.

\bigskip

So the gap to be filled in consists of finding analogues of Lemmas 5.2 and 5.5 of \cite{Bilu2017}.  In turn these rely on getting an analogue of Theorem 1.2 of \cite{Allombert2015}, which is where we finally fall flat.  This is a rather significant gap in the approach.  The theorem in question says that pairs of singular moduli can only be linearly independent over $\mathbb{Q}$ if they have degree at most 2.  This is the culmination of a significant amount of work from \cite{Allombert2015}, and the author has not yet had the time to work through whether a suitable analogue holds for $\chi^*$.  So here we must stop.

Before leaving this entirely, however, we will make the obvious comment; we earlier decided to ignore the presence of (\ref{eqn:XiqExpMatrix}), which does not appear in \cite{Bilu2017}.  In doing this, of course, we potentially throw away a significant amount of valuable information which could be used to our benefit.  A viable approach to bridge the gap in proving Conjecture \ref{conj:collinearChiPoints}, therefore, might be to make use of equation (\ref{eqn:XiqExpMatrix}) to circumvent the necessity of appeals to the missing Lemmas 5.2 and 5.5.  Once again, though, the author has not been able to work this through properly, so I will leave it here, in the hope that I have achieved my goal of laying out a viable approach to proving Conjecture \ref{conj:collinearChiPoints} and describing what remains to be done.

%If $\tau$ is a quadratic number, and $\tau'$ is another, of the same discriminant there is a Galois automorphism (over $\mathbb{Q}$) which sends $j(\tau)$ to $j(\tau')$.  In particular, if $\tau$ is quadratic with discriminant $-D$, there is a Galois automorphism $\theta$ over $\mathbb{Q}$ which sends $j(\tau)$ to $j\left(\frac{D+\sqrt{-D}}{2}\right)$.  By REF, Lemma REF, it follows that $\theta(\chi^*(\tau))=\chi^*\left(\frac{D+\sqrt{-D}}{2}\right)$ also.
%
%So if $p(j(\tau),\chi^*(\tau))$ vanishes, then so does $p\left(j\left(\frac{D+\sqrt{-D}}{2}\right),\chi^*\left(\frac{D+\sqrt{-D}}{2}\right)\right)$, which together with the bounds on $\Imm\tau$ above yields 
%\[\sqrt{D}\leq 2\max\left(2, \frac{6\operatorname{H}(p)c_1}{\pi}, (1-2\pi)^{-1}\log\left(\frac{2\pi^kk!\operatorname{H}(p)c_2}{3^k}\right) \right),\]
%as desired.
%
%
%If instead of over $\mathbb{Q}$, $p$ is defined over some number field $K$, the argument essentially does not change.  Suppose that $\tau$ has discriminant $-D$ and that $p(j(\tau),\chi^(\tau))$ vanishes.  Let $\theta$ be a Galois automorphism (over $\mathbb{Q}$) sending $j(\tau)$ to $j(\tau')$.  Then $\theta(p)(j(\tau'),\chi^*(\tau'))$ also vanishes.  Since there is always such a $\theta$ sending $j(\tau)$ to $j\left(\frac{D+\sqrt{-D}}{2}\right)$, we can consider $\theta(p)$ instead of $p$ and hence assume that $\tau=\frac{D+\sqrt{-D}}{2}$.

\bibliographystyle{scabbrv}
\bibliography{thebib}

\end{document}